\documentclass[12pt,a4paper]{amsart}

\usepackage{amsmath,amsthm,amsfonts,amssymb,mathrsfs}
\usepackage{subfigure} 
\usepackage{graphicx}
\usepackage{color}

\usepackage{array}
\usepackage{booktabs}
\usepackage{caption}
\usepackage{float}

\newcommand{\ud}{\mathrm{d}}

\usepackage{amsmath}
\usepackage{amssymb}
\usepackage{epsfig}
\usepackage{float}
\usepackage{latexsym,graphics}
\usepackage{dsfont}
\newtheorem{theorem}{Theorem}[section]
\newtheorem{proposition}[theorem]{Proposition}

\newtheorem{definition}[theorem]{Definition}

\newtheorem{remark}[theorem]{Remark}

\newcommand{\Cu}[1]{\mathbf{C^{#1}}}
\newcommand{\BC}{\mathbf{BC}}	
\newcommand{\Cl}[1]{\mathbf{C_{#1}}}			% C with lower index
			
\newcommand{\BL}{\mathbf{BL}}	 			% Bounded Lipschitz
\newcommand{\TV}{\mathbf{TV}}	
\newcommand{\Lip}{\mathbf{Lip}} 				% Lipschitz
\newcommand{\reali}{\mathbb{R}}
\newcommand{\naturali}{\mathbb{N}}			% Natural numbers

\def\sqr#1#2{{\vcenter{\vbox{\hrule height.#2pt\hbox{\vrule width.#2pt height#1pt \kern#1pt\vrule width.#2pt}\hrule height.#2pt}}}}

\renewcommand{\d}[1]{\mathrm{d}{#1}}

\newcommand{\norma}[1]{{\left \|#1\right \|}} 		
\renewcommand{\phi}{\varphi}

\allowdisplaybreaks[4]

\newtheorem{assumptions}[theorem]{Assumptions}
\newcommand{\R}{\mathbb{R}}
\newcommand{\M}{\mathcal M} 

\begin{document}

%%%%%%%%%%%%%%%%%%%%%%%%%%%%%
%%%%%%%%%%%%%%%%%%%%%%%%%%%%
\title{Asymptotic behaviour of a structured population model on a measure space}
\author[J. Z. Farkas]{J\'{o}zsef Z. Farkas$^1$}
\author[P. Gwiazda]{Piotr Gwiazda$^2$}
\author[A. Marciniak-Czochra]{Anna Marciniak-Czochra$^3$}
\address[$^1$]{Division of Computing Science and Mathematics, University of Stirling, Stirling, FK9 4LA, UK}
\email{jozsef.farkas@stir.ac.uk}
\address[$^2$]{Institute of Applied Mathematics and Mechanics,  University of Warsaw, Warszawa 02-097, Poland}
\email{pgwiazda@mimuw.edu.pl}
\address[$^3$]{Institute of Applied Mathematics,  Interdisciplinary Center for Scientific Computing (IWR) and BIOQUANT, University of Heidelberg, 69120 Heidelberg, Germany}
\email{anna.marciniak@iwr.uni-heidelberg.de}

\date{}

\begin{abstract}
In this paper we consider a physiologically structured population model with distributed states at birth, formulated on the space of non-negative Radon measures. Using a characterisation of the pre-dual space of bounded Lipschitz functions, we show how to apply the theory of strongly continuous positive semigroups to such a model. In particular, we establish the exponential convergence of solutions to a one-dimensional global attractor. \\


\end{abstract}

\maketitle

\section{Introduction}

Mathematical models describing the time evolution of physiologically structured populations have been extensively studied for many years, see e.g. \cite{MetzDiekmann,Webb}. Many of the models take the form of evolutionary partial differential equations describing the density of individuals with respect to a specific structural variable representing, for example,  age, size, phenotypic trait or cell maturity, see e.g. \cite{AcklehFarkas, CalsinaCuadrado, CalsinaCuadrado2, Desvillettes, Perthame}. Classical results for such models were obtained typically in the space of Lebesgue integrable functions (densities), see for example the early monograph \cite{Webb}. The choice of the state space $L^1$ is biologically motivated, as convergence with respect to the natural norm implies convergence of the total population size. At the same time, the spectral theory of positive operators on abstract Lebesgue (AL) spaces, which provides a convenient framework to analyse structured population models, is well-developed, see e.g. \cite{Arendt,Clement,EngelNagel,Sch}.

The asymptotic behaviour of solutions of age-structured models was already stu\-died using a direct Laplace transform technique, utilising the corresponding Volterra integral equation, see e.g. \cite{Feller, Feller2}. This approach was further developed later by a number of researchers, see e.g. \cite{I}. However, with the blossoming of semigroup theory in the middle of the $20$th century, researchers started to develop a general semigroup framework for treating structured population models, see e.g. \cite{EngelNagel,Webb}, and further references therein. For example, a number of authors have used semigroup theory to show that the long term behaviour of solutions of wide classes of models can be characterised via a reduction to a one dimensional evolution problem. In particular, compactness and positivity properties of the governing semigroups were proved to be useful to investigate the asymptotic behaviour of solutions, see e.g. \cite{WebbGrabosch,Webb1987,Webb}. Later, some of the results were generalised to other classes of physiologically structured population models, most notably size-structured models,  see e.g. \cite{MetzDiekmann,Thieme1998DCDS,Thieme1998JMAA}, and attempts were made to establish similar results for nonlinear models too, see e.g. \cite{GyllenbergWebb}. 

An alternative way of studying the long term dynamics of structured models have been proposed by Perthame and colleagues, see e.g. \cite{MichelMischlerPerthame,MichelMischlerPerthame2,MischlerPerthameRyzhik,Perthame}. This approach is based on multiplying the governing equations by a nonlinear function of the solution, the so called  entropy function, which in turn leads to a family of nonlinear renormalisations (relative entropies). This method can be directly applied only in the $L^1$ setting, and an extension to measure spaces requires a concept of a composition of a nonlinear function with bounded measures. Such an approach has been recently undertaken in \cite{GwiazdaWiedemann}.  At the same time, the method of generalised relative entropy does not require to establish a spectral gap condition (often utilised in the semigroup framework), but this in turn leads to the lack of the exponential convergence of solutions; instead, only algebraic convergence rate can be shown. 

More recently, there have been a number of results published showing that particular classes of quasilinear equations exhibit blow-up phenomena. That is, due to the nonlinear transport term, solutions with $L^1$ initial data tend to concentrate in finite time, see for example \cite{Ackleh1}. Therefore, researchers have started to study structured population models in spaces of non-negative Radon measures, see e.g. \cite{BGMC,CCGU,EversHilleMuntean,GLMC, GMC,Piccoli,two_sex}. 
This more general approach is particularly useful for mo\-dels, which are shown to exhibit a lack of local/global existence on Lebesgue spaces. 
At the same time, the idea of representing a heterogeneous population as a sum of masses concentrated in different points of the individual state space, can be also motivated from the biological point of view. For example, in some concrete applications, population data might be only obtained in measurements taken at discrete phy\-si\-olo\-gi\-cal states (cohorts); which data then naturally gives rise to define initial conditions in measure spaces. In particular, the possible choice of spaces of non-negative Radon measures was already proposed in \cite{MetzDiekmann}, as potentially relevant for biological applications, when the initial distribution of individuals is concentrated with respect to the structuring variable, i.e. it is not absolutely continuous with respect to the Lebesgue measure. 

A framework for the analysis of solutions of structured population models using Wasserstein-type metrics, adjusted to the non-conservative character of the considered problem, has been proposed in \cite{GLMC}, using a flat metric (bounded Lipschitz distance); and in \cite{GMC} using a Wasserstein-type metric, adjusted to spaces of non-negative Radon measures with integrable first moments.  The advantage of the proposed approach is in providing the structure of the space appropriate to compare solutions and to study their stability. Among others, continuous dependence with respect to the model ingredients is important in the context of numerical approximations and model calibration based on experimental data. However, so far, only results on existence, uniqueness and Lipschitz continuous dependence of solutions on the model ingredients have been obtained. This has allowed researchers to establish the stability of numerical schemes based on a particle method, for example the EBT (escalator boxcar train) algorithm, see e.g. \cite{EBT,CGU,GJMCU}.
Asymptotic analysis of a structured model with single state at birth, incroporating identical fertility and mortality rates; with initial data in a space of Radon measures equipped with the total variation norm has been undertaken recently in \cite{Gabriel}.
 
Models containing only integro-differential terms have been investigated on measure spaces using an approach of strongly continuous semigroups, see e.g. \cite{BuergerBomze}. The long-term behaviour of such structured populations models is also a topic of a recent paper by Mischler and Scher \cite{MischlerScher}, providing a spectral analysis of $C_0$-semigroups. Extending this approach to more general structured population models, i.e. to those including also a transport term on spaces of non-negative Radon measures leads to a difficulty, which is related to the lack of strong continuity of the semigroup generated by a transport equation on the space of measures with respect to the total variation norm.

In this paper we extend the semigroup approach applied previously on Lebesgue spaces, see e.g. \cite{FarkasHagen,FarkasHinow}, to a problem with initial data in a space of non-negative Radon measures. Results concerning Lipschitz dependence on time, initial data and model parameters of the solutions in the space of measures were based on the analysis of semi-flows in metric spaces given by a positive cone of Radon measures, equipped with a Lipschitz bounded distance, or other Wasserstein-type metrics. 

When approaching the problem with the spectral theory of positive  semigroups, one needs a Banach lattice. Extending the positive cone to a whole space of Radon measures with a Lipschitz bounded norm provides a linear structure of the functional space, which needs to be closed to obtain a Banach space. The Lipschitz bounded norm has been equivalently defined in the literature as flat norm \cite{Neunzert}, Kantorovich-Rubinstein norm \cite{Bogachev}, Fortet-Mourier norm \cite{Lasota} or Dudley norm \cite{EversHilleMuntean}, see \cite{GMT} for more details.  

To establish irreducibility of a semigroup, it may be crucial to know a cha\-rac\-te\-risation of the pre-dual space to the space of bounded Lipschitz functions. In the abstract theory of functional analysis and approximation theory a so-called free Lipschitz space is considered. This space is pre-dual to the space of Lipschitz continuous functions with $0$ at origin, denoted by $Lip_0$, see e.g. \cite{Weaver}. Similar constructions of the pre-dual space to the space of bounded Lipschitz functions, which are not necessarily zero at a fixed point, have been proposed by Hille and Worm in \cite{HilleWorm}. A remarkable observation is that the transport semigroup is strongly continuous on this pre-dual space and that this space is a closure of the space of Radon measures with respect to the bounded Lipschitz distance. This allows the extension of some of the results on strongly continuous quasi-compact, positive semigroups from the $L^1$ setting to more general state spaces. Thus, we will prove that asymptotically our model essentially reduces to a one dimensional evolution problem, under some assumptions on the model ingredients. Interestingly, the $C_0$ property of the transport semigroup does not hold in the dual space, as shown in \cite[Lemma 2]{GOU}. 

We also note that our results can be extended to a more general setting, in particular to a space of real Borel measures on a metric space, see \cite{GMT}. Such framework can then be applied to analyse structured population models on networks.

\section{Spaces and norms}
Here we briefly introduce the notations for the spaces and norms we are going to use throughout the paper. In particular we  set our state space as
\begin{equation*}
\mathcal{X}=\overline{\mathcal{M}(\mathbb{R}_+)}^{||\cdot ||^*_{\BL}};
\end{equation*}
that is, $\mathcal{X}$ is the closure of the set of Radon measures on $\mathbb{R}_+$, with respect to the norm
\begin{equation*}
||\mu||_{\BL}^*:=\sup\left\{\left|\int\phi\,\ud \mu\right|\,:\,||\phi||_{\BL}\le 1\right\}.
\end{equation*}
Here $||\cdot||_{\BL}^*$ is a dual norm to the norm on the space $\BL(\mathbb{R}_+)$ of bounded Lipschitz functions on 
$\mathbb{R}_+$, given by
\begin{equation}
\|\phi\|_{\BL}:=\left|\left| \,|\phi | + \left|\partial_x \phi\right|\, \right|\right|_{\infty}.\label{ourBL}
\end{equation}

Note that $\left(\BL(\mathbb{R}_+),||\cdot||_{\BL}\right)$ is a Banach lattice, with positive cone \begin{equation*}
\BL_+(\mathbb{R}_+):=\left\{f\in \BL(\mathbb{R}_+)\,|\,f\ge 0\right\},
\end{equation*} 
which defines an ordering, i.e. $f\ge g$ if and only if $(f-g)\in\BL_+(\mathbb{R}_+)$.  
Furthermore, we have
\begin{equation*}
\left(\overline{\mathcal{M}(\mathbb{R}_+)}^{||\cdot ||^*_{\BL}}\right)^* = \BL(\mathbb{R}_+),
\end{equation*} 
see Theorem 3.7 in  \cite{HilleWorm}, and in particular, in our setting we have
\begin{equation*}
\left(\overline{\mathcal{M}(\mathbb{R}_+)}^{||\cdot ||^*_{\BL}}\right)_+={\mathcal M_+(\mathbb{R}_+)}=\mathcal{X}_+,
\end{equation*} 
see Theorem 3.9 in  \cite{HilleWorm}.

The latter two properties hold, since the  norm chosen above in $\eqref{ourBL}$ is equivalent to the norm used in \cite{HilleWorm}, i.e. it is equivalent to the norm
\begin{equation*}
\|\phi\|_{\BL2}:= ||\phi\|_{\infty} + |\phi|_{\Lip},
\end{equation*}
where 
\begin{equation*}
|\phi|_{\Lip}:=\displaystyle\sup_{x,y \in \mathbb{R}_+}\left\{ \frac{\big|\phi(x)-\phi(y)\big|}{d(x,y)}, x \not= y\right\},
\end{equation*}
and  
\begin{equation*}
\|\phi\|_{\infty}:=\displaystyle\sup_{x\in\reali_+}\left|\phi(x)\right|.
\end{equation*} 
Our choice of the norm $\eqref{ourBL}$ is important for the analysis presented in Section 4.

\section{Problem formulation and existence of solutions}

We consider a linear structured population model involving a transport operator describing the development of individuals with respect to a physiological structuring variable (determining individual state), an integral operator describing the birth/recruitment process, and a li\-ne\-ar decay term accounting for individual mor\-ta\-li\-ty. Specifically, we consider the following model.
\begin{equation}
\partial_t \mu+\partial_x \left(b(x) \, \mu \right)+c(x)\,\mu = \displaystyle\int_{\R_+} \eta(y)\, \ud\mu(y),\quad (t,x) \in  \R_+ \times \R_+,\quad  \mu_0 \in  \mathcal{X}_+. \label{Model}
\end{equation}
Linear and nonlinear structured population models with distributed recruitment processes, but formulated on Lebesgue spaces were introduced and studied recently for example in \cite{AcklehFarkas,CDF,FarkasHinow}. However, it is important to note that   in contrast to \cite{AcklehFarkas,CDF,FarkasHinow} here we do not impose a finite maximal value for the structuring variable $x$. For some applications this might be a more natural assumption; but at the same time this poses additional challenges in the spectral analysis of model \eqref{Model}. This was already observed in the case of different classes of models formulated on Lebesgue spaces, see e.g. \cite{FarkasHagen}.

%\begin{eqnarray}
%\partial_t \mu_t + \rm{(b \mu_t)= \int_{\Omega}K(x,y) \mu_t (dy) - c\mu_t,
%\end{eqnarray}

We impose the following assumptions on the model parameters.
\begin{assumptions}\label{Assum} 
\ \begin{itemize}
\item[(i)] $c \in  \BL(\R_+)$.
\item[(ii)] $y\mapsto \eta(y) \in  \BL\left(\R_+; (\mathcal{X}_+, \norma{\cdot}_{\BL}^*)\right)$.
\item[(iii)] $b \in  \BL(\R_+)$, $b > 0$.
\end{itemize}
\end{assumptions}
Note that these assumptions are required to study the existence and uniqueness of solutions of model \eqref{Model}. Later on,  when studying the asymptotic behaviour of solutions, we will impose further conditions on the model ingredients.

We note that a norm in the space $\BL\left(\R_+; (\mathcal{X}_+, \norma{\cdot}_{\BL}^*)\right)$ is defined as
\begin{displaymath}
\norma{\eta}_{\BL} = \norma{\eta}_{\BC} + \Lip(\eta) \, ,
\quad \mbox{where } \quad
\norma{\eta}_{\BC} =
\sup_{x\in\reali_+}\norma{\eta(x)}_{\BL}^*, \,
\end{displaymath}
and $\Lip(\eta)$ is the usual Lipschitz constant of $\eta$.

We begin with the definition of a solution of model \eqref{Model} on a finite time interval $[0,T]$, and for values in a positive cone of Radon measures.

\begin{definition}
 \label{def:WeakSolution}
Given $T>0$, a function $\mu \colon [0,T] \to\mathcal{X}_+$ is a \emph{weak solution}
of model \eqref{Model} on the time interval $[0,T]$, if $\mu$ is narrowly continuous with respect to time, and for all $\phi \in (\Cu{1}
 \cap \BL) \left(\reali_+\times \reali_+\right)$ the following equality holds:
 \begin{eqnarray}
  && \int_0^T \int_{\reali_+} \left(
   \partial_t \phi(t,x) + b(x) \; \partial_x
   \phi(t,x) - c(x) \; \phi(t,x) \right)\,\ud\mu_t(x)\,\ud{t} \nonumber
  \\
  \label{Formulation:WeakSolution}
  && \quad
  + \int_0^T \int_{\reali_+} \left( \int_{\reali_+} \phi(t,y)
   \d{\left[\eta(x)\right]} (y) \right)\, \ud\mu_t(x)\,\ud{t}
  \\
  && = \int_{\reali_+} \phi(T,x) \; \d{\mu_T}(x) - \int_{\reali_+}
  \phi(0,x) \; \d{\mu_0}(x).
  \nonumber
 \end{eqnarray}
\end{definition}
\noindent
For the notion of \emph{narrow continuity} we refer to \cite[\S~5.1]{Ambrosio}, where this concept was introduced. 
\begin{definition}\label{narrow_conv}
We say that a sequence $\{\mu^n\}_{n\in\naturali} \subset \mathcal M(\reali_+)$ converges narrowly to a measure $\mu \in \mathcal M(\reali_+)$, if and only if 
\begin{equation*}
\lim_{n \to +\infty} \int_{\reali_+} \phi(x) \,\ud\left(\mu^n - \mu\right) (x) = 0, \quad \forall\phi \in \Cl{b}(\reali_+),
\end{equation*}
where $\Cl{b}$ denotes the space of bounded continuous functions.
Similarly, we say that a mapping $\mu : [0,T] \mapsto {\mathcal M}(\reali_+)$ is narrowly continuous, if for every $\phi \in \Cl{b}(\reali_+)$ the function
$$
f:[0,T]\mapsto \reali,\quad f(t) = \int_{\reali_+} \phi(x) \,\ud \mu_t(x)
$$
is continuous.
\end{definition}

Above, in Definition \ref{def:WeakSolution}, the integral $\int_{\reali_+} \phi(t,y) \,\ud{\left[\eta(x)\right]}(y)$ denotes the integral of $\phi(t,y)$ with respect to the measure $\eta(x)$ in the variable $y$. Similarly, $\int_{\reali_+}\phi(T,x) \,\ud{\mu_T}(x)$ is the integral of $\phi(T,x)$ with respect to the measure $\mu_T$ in the variable $x$.

Local-in-time existence of solutions and their Lipschitz dependence  on time, initial data and model parameters follow from Theorem 2.10 in \cite{CCGU},  formulated therein for the non-autonomous case.  For the reader's convenience, we recall this result adjusted to our problem and replacing the flat metric by the bounded Lipschitz norm. 

\begin{proposition} \label{existenceAgnieszka}
Let  Assumptions \ref{Assum} (i)-(iii) hold true. Then model \eqref{Model} is governed by a strongly continuous semigroup  $\left\{\mathcal{T}(t)\right\}_{t\geq0}$ on a finite time interval $[0,T]$, that admits the following properties.
 \begin{enumerate}
 \item $\mathcal{T}(0) = \mathrm{\mathbf{Id}}$, and for all $t_1, t_2 \in [0,T]$ with $t_1+t_2 \in
  [0,T]$, we have $\mathcal{T}(t_1) \circ \mathcal{T}(t_2) = \mathcal{T}(t_1+t_2)$.
 \item For all $t \in [0,T]$ and for all $\mu_1, \mu_2 \in
  \mathcal{X}_+$, the following estimate holds:
  \begin{displaymath}
    \norma{\mathcal{T}(t)\,\mu_1-\mathcal{T}(t)\,\mu_2}^*_{{\BL}}
   \leq
C_1(t)
  \left|\left|\mu_1-\mu_2\right|\right|^*_{\BL},
  \end{displaymath}
  where 
  $$C_1(t)=   \exp
   \big[
    3t
    \left(
     \norma{\partial_x b}_{\infty}
     +
     \norma{c}_{\BL}
     +
     \norma{\eta}_{\BL}
    \right)\big].$$
 \item For all $t \in [0,T]$ and for all $\mu_0 \in
  \mathcal{X}_+$, define $\mu_t = \mathcal{T}(t)\,\mu_0$. Then, the
  solution $\mu$ of problem \eqref{Model} is Lipschitz continuous with respect to time and the following estimate
  holds:
  \begin{displaymath}
   \norma{\mathcal{T}(t)\,\mu_0 - \mu_0}^*_{{\BL}}
   \leq C_2(t)
  \norma{\mu_0}_{\TV},
  \end{displaymath}
  where  $  \norma{\cdot}_{\TV}$ denotes the total variation norm, and
  $$ C_2(t) =  
    \norma{b}_\infty
    +
    \left(
     \norma{c}_{\infty}
     +
     \norma{\eta}_{\BC}
    \right)
    \exp{ \left[
      \left(
       \norma{c}_{\infty} + \norma{\eta}_{\BC}
      \right)
      t \right]}.$$
 \item For all $\mu_0 \in \mathcal{X}_+$, the orbit $t \to
  \mathcal{T}(t)\,\mu_0$ of the semigroup is a weak solution of the linear autonomous
  problem~\eqref{Model} in the sense of
  Definition~{\rm\ref{def:WeakSolution}}.
 \end{enumerate}
\end{proposition}

Above, in Proposition \ref{existenceAgnieszka} assertion (2) corresponds to the Lipschitz dependence of a model solution on the
initial data, while assertion (3) characterises its time regularity. In the next section we show that the semigroup $\mathcal{T}(t)$ can be extended to the whole time interval $[0,\infty)$ due to the arbitrary choice of $T<\infty$.

\section{Asymptotic behaviour}

In this section we are going to characterise the asymptotic behaviour of solutions of model \eqref{Model}. 
In particular, using results from the theory of strongly continuous po\-si\-tive semigroups on Banach lattices, we show that solutions of model \eqref{Model} approach a finite dimensional attractor. One of the main difficulties we need to 
overcome is that model \eqref{Model} is not governed by an eventually compact semigroup; and the lack of eventual compactness in turn poses challenges in the spectral analysis of the semigroup. We refer the interested reader to \cite{FarkasHagen}, where a  (hierarhic) structured population model, where individuals exhibit cannibalistic behaviour, with an unbounded individual state space (i.e. no finite maximal size)  was investigated. In contrast to the model studied in \cite{FarkasHagen} an added inherent difficulty of model \eqref{Model} is that for structured population models with distributed recruitment processes it is not possible to characterise the point spectrum of the generator of the semigroup via roots of an associated characteristic equation, in general; see for example  \cite{FarkasHinow} for more details. 

For some standard definitions and notations from the spectral theory of strongly continuous semigroups not explicitly introduced here we refer the reader to  \cite{EngelNagel}.

First we rewrite model \eqref{Model} as an abstract Cauchy problem on the state space $\mathcal{X}=\overline{\mathcal{M}(\mathbb{R}_+)}^{||\cdot ||^*_{\BL}}$ as follows:

\begin{equation}\label{Cauchy}
\frac{\ud \mu}{\ud t}=\left(\mathcal{A}+\mathcal{B}+\mathcal{C}\right)\mu, \quad \mu(0)=\mu_0,
\end{equation}
where we define
\begin{align}
\mathcal{A}\,\mu&=-\frac{\partial}{\partial x}\left(b\,\mu\right),  \hspace{15mm} D(\mathcal{A})=L^1(\mathbb{R}_+)\cap \mathcal{X}, \label{Adef} \\
\mathcal{B}\,\mu&=-c\,\mu, \hspace{25mm}  D(\mathcal{B})=\mathcal{X}, \\
\mathcal{C}\,\mu&=\int_{\mathbb{R}_+}(\eta(y))\,\ud \mu(y), \hspace{5mm}   D(\mathcal{C})=\mathcal{X}.\label{Cdef}
\end{align}
Note that $\mathcal{A}$ is a densely defined closed operator, and the assumptions we imposed on $b$ (see Assumptions \ref{Assum}) imply that it generates a strongly continuous semigroup of positive operators on $\mathcal{X}$, denoted by $\mathcal{T}_{\mathcal{A}}(t)$. 
Furthermore, our assumption on $\eta$ (see Assumptions \ref{Assum} (ii)) implies that $\mathcal{C}$ is a positive operator, i.e. it maps $\mathcal{X}_+$ into $\mathcal{X}_+$.

\begin{proposition}\label{stronglycontinuous}
Under Assumptions \ref{Assum}(i)-(iii), there exists a unique weak solution $\mu \colon [0,T] \to
\mathcal{X}$ of model \eqref{Model}, which coincides with a trajectory of a strongly continuous semigroup $\mathcal{T}(t)$ on the Banach space $\overline{\mathcal{M}(\mathbb{R}_+)}^{||\cdot ||^*_{\BL}}$, defined for all $t\in \mathbb{R}_+$.
\end{proposition}
\begin{proof}
The assertion for the positive cone $\mathcal{X}_+$ follows from Proposition \ref{existenceAgnieszka} and from the observation that, in fact, the semigroup $\mathcal{T}(t)$ can be defined on the whole interval $[0,\infty)$, due to the arbitrary choice of $T<\infty$.  Indeed, two semigroups $\mathcal{T}^{T_1}(t)$ and $\mathcal{T}^{T_2}(t)$, defined for $t\in [0,T_1]$ and $t\in [0,T_2]$, respectively,  coincide on the interval $[0, \min\{T_1,T_2\}]$, because the corresponding solutions satisfy the weak formulation. Hence, the semigroup $\mathcal{T}(t)$ of solutions can be extended to the whole interval $[0,\infty)$.
Also note that the semigroup $\mathcal{T}(t)$ is Lipschitz with respect to time and initial data. To extend the result to the whole state space $\M(\R_+)$, we apply the Hahn-Jordan decomposition of a measure into its negative and positive part, and use the linearity of the problem.  Hence, a strongly continuous semigroup is defined for all initial data in the normed space $\mathcal M(\R_+)$ with the bounded Lipschitz distance. Also note that a Lipschitz operator can be extended to the closure of the domain \cite[Th.2.6]{Amman} $\overline{\mathcal{M}(\mathbb{R}_+)}^{||\cdot ||^*_{\BL}}=\mathcal{X}$.

To prove that it defines a strongly continuous semigroup for $t=0$, on the whole state space $\mathcal{X}$, we take an approximation of $\mu_0 \in \mathcal{X}$ by Radon measures $\mu_0^{\varepsilon}$, such that
$$ \norma{\mu_0 - \mu_0^{\varepsilon}}^*_{{\BL}}<\frac{\varepsilon}{2C_1(1)+1}$$,
with $C_1$ given in  Proposition \ref{existenceAgnieszka}. Using the estimates (2) and (3) from Proposition \ref{existenceAgnieszka}, we obtain
\begin{eqnarray*}
 \norma{\mathcal{T}(t)\,\mu_0 - \mu_0}^*_{{\BL}} \leq  \norma{\mathcal{T}(t)\,\mu_0 - \mathcal{T}(t)\,\mu_0^{\varepsilon}}^*_{{\BL}} + \norma{\mathcal{T}(t)\,\mu_0^{\varepsilon} - \mu_0^{\varepsilon}}^*_{{\BL}}+ \norma{\mu_0 - \mu_0^{\varepsilon}}^*_{{\BL}} <\varepsilon
\end{eqnarray*}
for 
\begin{equation*}
t<\delta= \frac{\varepsilon}{2\,C_2(1)\,\norma{\mu_0^{\varepsilon}}_{\TV}}.
\end{equation*}
\end{proof}

\begin{remark}
Note that the choice of the space $\mathcal X$ is essential. The strong continuity does not hold if we take the dual space $(\BL)^*$ instead of its closed subset $\mathcal X$, see \cite[Lemma 2]{GOU}.
\end{remark}

Next we are going to characterise the asymptotic behaviour of the semigroup $\mathcal{T}(t)$ generated by $\mathcal{A}+\mathcal{B}+\mathcal{C}$, and in turn the asymptotic behaviour of solutions of model \eqref{Model}. 
The asymptotic behaviour of $\mathcal{T}(t)$ is naturally determined by its growth bound $\omega_0$, together with the boundary spectrum of its generator (in the simplest case, its spectral bound). Let us recall the definition of the growth bound of a semigroup $\mathcal{T}(t)$, and the spectral bound of its generator $\mathcal{A}$.
\begin{equation*}
\omega_0=\omega_0(\mathcal{T}):=\inf\left\{w\in\mathbb{R}\,|\,\exists\, M_w\ge 1,\,\text{such that}\,  ||\mathcal{T}(t)||\le M_w e^{wt},\,\forall\,t\in\mathbb{R}_+ \right\},
\end{equation*}
\begin{equation*}
s(\mathcal{A})=\sup\left\{\text{Re}(\lambda)\,|\,\lambda\in\sigma(\mathcal{A})\right\},
\end{equation*}
where $\sigma(\mathcal{A})$ denotes the spectrum of the operator $\mathcal{A}$. Note that in general we have $-\infty\le s(\mathcal{A})\le\omega_0<\infty$. 

Furthermore, we recall that for a bounded linear operator $T$ on $\mathcal Y$, the so-called essential norm
is given by
\begin{equation}
	\|T\|_\text{ess}:=\text{dist}\,\left(T, K({\mathcal Y})\right),
\end{equation}
where $K({\mathcal Y})$ denotes the set of compact linear operators on $\mathcal Y$. 
The essential growth bound of the semigroup $\mathcal{T}(t)$ on $\mathcal Y$ with generator ${\mathcal A}$ is then defined as
\begin{equation}\label{ess1}
	\omega_\text{ess}\left(\mathcal{T}\right)\left[=\omega_\text{ess}\left(\mathcal{A}\right)\right]=\lim_{t\rightarrow \infty} \left(\frac{\ln \|\mathcal{T}(t)\|_\text{ess}}{t}\right).
\end{equation}
It is readily seen that, for any ${\mathcal K}\in K({\mathcal Y})$,
\begin{equation}\label{ess2}
	\omega_\text{ess}\left({\mathcal A}\right)=\omega_\text{ess}\left({\mathcal A}+{\mathcal K}\right).
\end{equation}
The significance of the essential growth bound lies in the fact that
\begin{equation}
	\omega_0\left({\mathcal T}\right)= \max\,\left\{\omega_\text{ess}\left({\mathcal T}\right), s\left({\mathcal A}\right)\right\},
\end{equation}
see \cite{EngelNagel} for more details.

There is an extensive literature of results on the asymptotic behaviour of positive semigroups, and in particular for positive irreducible semigroups exhibiting in addition certain compactness properties, see e.g. \cite{Arendt, Clement, EngelNagel}. Indeed, over the past couple of decades results concerning spectral properties of semigroups naturally arising in structured population dynamics became abundant, we just mention here the early papers \cite{Thieme1998DCDS,Thieme1998JMAA,Webb1987}. 
In paticular, we note that for eventually compact semigroups (which naturally govern structured population models with bounded individual state space) the spectral mapping theorem holds true, and in particular their essential spectrum is empty, see e.g. \cite{EngelNagel,FH}. Our model \eqref{Model} does not impose a finite maximal value of the structuring variable, hence the governing semigroup cannot be shown eventually compact. To overcome the lack of this desirable regularity property of the governing semigroup we are going to apply a different notion of compactness. In particular, let us recall from \cite{EngelNagel} the notion of quasi-compactness of a strongly continuous semigroup. 
\begin{definition}
A strongly continuous semigroup $\mathcal{T}(t)$ on a Banach space $\mathcal{X}$ is called quasi-compact if
\begin{equation*}
\lim_{t\rightarrow\infty} \inf \left\{\norma{\mathcal{T}(t)-\mathcal{K}}\,\,\vert\,\, \mathcal{K}\in K(\mathcal{X})\right\}=0.
\end{equation*}
\end{definition}
Note that, as stated in Proposition 3.5. in \cite[Ch.V]{EngelNagel}, a semigroup $\mathcal{T}(t)$ is quasi-compact if and only if $\omega_{ess}(\mathcal{T})<0$. 
\begin{proposition}\label{quasi-compact}
Let  Assumptions \ref{Assum}(i)-(iii) hold true, and in addition assume  that $b'(x) \leq 0$ holds; and that there exists a constant  $\kappa>0$, such that for every $x \in \R_+$ we have $c(x) \geq \left|c'(x)\right| + \kappa$. Then, the semigroup $\mathcal{T}(t)$ generated by $\mathcal{A}+\mathcal{B}+\mathcal{C}$ is quasi-compact.
\end{proposition}
\begin{proof}
First note that the integral operator $\mathcal{C}$ is compact, since it can be approximated by operators of finite dimensional range. Hence by invoking Proposition 3.6 from \cite[Ch.V]{EngelNagel} it is sufficient to show that the semigroup $\mathcal{T}_{\mathcal{A}+\mathcal{B}}(t)$ generated by $\mathcal{A}+\mathcal{B}$ is quasi-compact. 

To this end note that on the grounds of  Proposition 3.5 in \cite[Ch.V]{EngelNagel} it is sufficient to show that the following inequality holds true 
\begin{equation}\label{quasi-ess}
\max\left\{s(\mathcal{A}+\mathcal{B}),\omega_{ess}\left(\mathcal{T}_{\mathcal{A}+\mathcal{B}}\right)\right\}=\omega_0\left(\mathcal{T}_{\mathcal{A}+\mathcal{B}}\right)<0;
\end{equation}
that is, the semigroup $\mathcal{T}_{\mathcal{A}+\mathcal{B}}(t)$ is strictly contractive.

To estimate the norm of the semigroup $\mathcal{T}_{\mathcal{A}+\mathcal{B}}(t)$, we consider the dual problem (backward equation)
 \begin{eqnarray}
   \partial_t \phi(t,x) = b(x) \; \partial_x
   \phi(t,x) + c(x) \; \phi(t,x).   \label{dual}
 \end{eqnarray}
The dual problem above is governed by the adjoint semigroup $\mathcal{T}_{\mathcal{A}+\mathcal{B}}^*(t)$, and we have 
\begin{equation}
\norma{\mathcal{T}_{\mathcal{A}+\mathcal{B}}}_{\mathcal L\left((\mathcal X, \norma{\cdot}_{\BL}^*);(\mathcal X, \norma{\cdot}_{\BL}^*)\right)} = \norma{\mathcal{T}^*_{\mathcal{A}+\mathcal{B}}}_{\mathcal L\left( \BL;\BL\right)}.
\end{equation}
In particular, we consider two problems given by the adjoint operators
\begin{align}
\mathcal{A^*}\,\phi&=  b \, \partial_x \phi, \label{A*} \\
\mathcal{B^*}\,\phi&=c\,\phi. \label{B*}
\end{align}

First we note that the semigroup $\mathcal{T}^*_{\mathcal{A}}(t)$ generated by the adjoint operator $\mathcal{A^*}$ is contractive (see e.g. \cite[Ch.II]{EngelNagel}), i.e. $\left|\left|\mathcal{T}^*_{\mathcal{A}}(t)\right|\right|\le 1,\,\forall\,t\ge 0$ holds, if $b$ is a monotone decreasing function, i.e. $b' \leq 0$. Consequently, it  follows from the definition of the growth bound $\omega_0$ (see above) that $\omega_0\left(\mathcal{T}^*_{\mathcal{A}}\right)\le 0$ holds. 

Next we show that the (bounded) operator $\mathcal{B}^*$ generates a positive contraction semigroup $\mathcal T^*_{\mathcal{B}}(t)$ satisfying $\norma{\mathcal T^*_{\mathcal{B}}(t)}\le e^{-\kappa t},\, t\ge 0$, for some $\kappa>0$. 

Since the corresponding backward equation holds for negative times, we change variables $\tau = -t$ and obtain
\begin{eqnarray*}
\partial_{t} \phi(t,x) = - c(x)  \phi(t,x).
\end{eqnarray*}
Differentiating with respect to $x$ both sides of the equation above yields
\begin{eqnarray*}
\partial_t \partial_x  \phi(t,x) = - c(x) \partial_x \phi(t,x)-  c'(x)\; \phi(t,x).
\end{eqnarray*}
We have the following estimates
\begin{eqnarray*}
\partial_t |\phi(t,x)| &\leq& - c(x)  |\phi(t,x)|,\\
\partial_t |\partial_x \phi(t,x)| &\leq& - c(x) |\partial_x \phi(t,x)| - c'(x) \; \phi(t,x) \;  \rm{sgn}(\partial_x \phi(t,x))\\
& \leq&   - c(x) |\partial_x \phi(t,x)| + |c'(x)| \; |\phi(t,x)|.
\end{eqnarray*}
Assuming that there exists a $\kappa>0$ such that $c(x) \geq   |c'(x)| + \kappa$ for every $x\in \R_+$, we obtain
\begin{eqnarray*}
\partial_t \left( |\phi(t,x)| + |\partial_x \phi(t,x)|\right)  &\leq& - \kappa \left( |\phi(t,x)| + |\partial_x \phi(t,x)|\right).\\
\end{eqnarray*}
Therefore we conclude that $\norma{\mathcal T^*_{\mathcal{B}}(t)}\le e^{-\kappa t},\, \forall\,t\ge 0$. Note that the operator $\mathcal B^*$ is bounded. It then follows from a version of the Trotter product formula (see e.g. Corollary 5.8  in \cite[Ch.III]{EngelNagel}) that the semigroup $\mathcal{T}^*_{\mathcal{A}+\mathcal{B}}(t)$ generated by $\mathcal{A}^*+\mathcal{B}^*$ satisfies 
\begin{equation}
\left|\left|\mathcal{T}^*_{\mathcal{A}+\mathcal{B}}(t)\right|\right|\le \exp\left(-\kappa\,t\right),\quad \forall\,t\ge 0,
\end{equation}
which implies
\begin{equation}
\left|\left|\mathcal{T}_{\mathcal{A}+\mathcal{B}}(t)\right|\right|\le \exp\left(-\kappa\,t\right),\quad \forall\,t\ge 0,
\end{equation}
and therefore we have $\omega_0\left(\mathcal{T}_{\mathcal{A}+\mathcal{B}}\right)\le -\kappa<0$, and the proof is completed.
\end{proof}

The significance of quasi-compactness of a semigroup $\mathcal{T}(t)$ is demonstrated by the following cha\-rac\-te\-ri\-sation theorem, recalled from \cite{EngelNagel} for the reader's convenience.
\begin{theorem}\cite[Ch.V, Theorem 3.7]{EngelNagel}\label{quasi-compact-char}
Let $\mathcal{T}(t)$ be a quasi-compact strongly con\-ti\-nuous semigroup with generator $\mathcal{A}$ on the Banach space $\mathcal{X}$. Then, the following holds.
\begin{itemize}
\item[(i)] The set $\left\{\lambda\in\sigma(\mathcal{A})\,|\, Re(\lambda)\ge 0\right\}$ is finite (possibly empty!), and consists of poles of the resolvent operator $R(\cdot, \mathcal{A})$ of finite algebraic multiplicity.
\item[(ii)] If we denote the set of poles by $\lambda_1,\cdots,\lambda_m$ and their corresponding residues by $P_1,\cdots,P_m$, with orders $k_1,\cdots,k_m$, respectively; then we have
\begin{align*}
\mathcal{T}=\mathcal{T}_1(t)+\cdots+\mathcal{T}_m+\mathcal{R}(t),
\end{align*}
where
\begin{align*}
\mathcal{T}_n(t)=e^{\lambda_n\,t}\displaystyle\sum_{j=0}^{k_n-1}\frac{t^j}{j!}(\mathcal{A}-\lambda_n)^j\,P_n,\quad t\ge 0,\quad 1\le n\le m,
\end{align*}
and
\begin{align*}
||\mathcal{R}(t)||\le M\,e^{-\varepsilon\,t},\quad \text{for some} \quad \varepsilon>0,\,M\ge 1,\quad \forall\,t\ge 0.
\end{align*}
\end{itemize}
\end{theorem} 

Moreover, if there exists a dominant eigenvalue $\lambda_*$ of multiplicity one with residue $P_*$, then 
\begin{equation}\label{AEG}
\left|\left|e^{-\lambda_*\,t}\,\mathcal{T}(t)-P_*\right|\right|\le M\,e^{-\varepsilon\,t},
\end{equation}
for some $M\ge 1$ and $\varepsilon>0$. This property of a semigroup $\mathcal{T}(t)$ is called asynchronous exponential growth (AEG for short), see e.g. \cite{Webb1987}. We  note that according to the definition of AEG in \cite{Webb1987}, the dominant eigenvalue $\lambda_*$ does not necessarily have to have algebraic multiplicity one. 

The existence of a dominant eigenvalue $\lambda_*$ of multiplicity one is guaranteed if the semigroup $\mathcal{T}(t)$ is irreducible, unless its spectrum is empty, see e.g. \cite[C-III, Proposition 3.5]{Arendt}.

Next we recall a definition of irreducibility (see e.g. \cite[C-III, Definition 3.1]{Arendt}), and formulate a sufficient condition for the irreducibility of the semigroup $\mathcal{T}(t)$ generated by $\mathcal{A}+\mathcal{B}+\mathcal{C}$. 
Below $\langle\cdot,\cdot\rangle$ stands for the semi-inner product, i.e. the natural pairing between elements of the Banach space $\mathcal{X}$ and its dual $\mathcal{X}^*$.

\begin{definition}\label{irr}
The positive semigroup $\mathcal{T}(t)$ on the Banach lattice $\mathcal{X}$ is called irreducible, if for all $0\not\equiv x\in \mathcal{X}_+,\,0\not\equiv x^*\in\mathcal{X}^*_+$ there exists a time $t>0$ such that $\langle\mathcal{T}(t)\,x,x^*\rangle >0$.
\end{definition} 
Note that in our setting we have $\langle \mu,\phi\rangle=\int_0^\infty\phi\,\ud \mu$ for $\mu\in\mathcal{X}_+,\,\phi\in\mathcal{X}^*_+$.
\begin{proposition}\label{irreducible}
If there exists a $\hat{y}>0$, such that $0\in\displaystyle\bigcap_{y>\hat{y}}\text{supp}\,(\eta(y))$, then the semigroup $\mathcal{T}(t)$ generated by $\mathcal{A}+\mathcal{B}+\mathcal{C}$ is irreducible. 
\end{proposition}
\begin{proof}
First note that the above characterisation of irreducibility (Definition \ref{irr}) is equivalent to the following condition: \begin{equation*}
\forall\,\mu\in\mathcal{X}_+,\,\mu\not\equiv 0\quad \text{we have}\quad \displaystyle\bigcup_{t\ge 0} \text{supp}(\mathcal{T}(t)\,\mu)=\mathbb{R}_+.
\end{equation*}
Next note that, since $b>0$ and $c$ is bounded (see Assumption \ref{Assum}); for every $\mu\in\mathcal{X}_+$ it holds that if for some time $t^*\ge 0$ one has $y^*\in \text{supp}(\mathcal{T}(t^*)\,\mu)$, then for all $y>y^*$ there exists a time $t\ge t^*$ such that $y\in\text{supp}(\mathcal{T}(t)\,\mu)$. 

Now take any $0\not\equiv \mu\in\mathcal{X}_+$. By the previous observation (abd since $\mu$ is not the zero element), there exists a $t^*\ge 0$, such that there exists a $y^*>\hat{y}$, such that we have $y^*\in\text{supp}(\mathcal{T}(t^*)\,\mu)$. Then, since $0\in \text{supp}\,\eta(y^*)$, it follows that for every $t>t^*$ we have $0\in \text{supp}(\mathcal{T}(t)\,\mu)$, and therefore by the observation above we have $\displaystyle\bigcup_{t\ge 0} \text{supp}(\mathcal{T}(t)\,\mu)=\mathbb{R}_+$.
\end{proof}
Note that from the biological point of view the irreducibility condition above is very natural, it requires (when interpreting the structuring variable $x$ as individual size for example) that large individuals produce offspring of minimal size, see e.g. \cite{CDF}.

We are now in the position to conclude by formulating our main result, describing the asymptotic behaviour of solutions of model \eqref{Model} in the following theorem.
\begin{theorem}\label{mainresult}
Let the assumptions of Proposition \ref{quasi-compact} hold true. Then, one of the following holds true.
\begin{itemize}
\item[(i)] The semigroup $\mathcal{T}(t)$ generated by $\mathcal{A}+\mathcal{B}+\mathcal{C}$ is uniformly exponentially stable, 
i.e. $\omega_0(\mathcal{T})<0$, that is, solutions of model \eqref{Model} tend to zero.
\item[(ii)] There exists a $\lambda_*\ge 0$ and a finite rank operator $P_*$ on $\mathcal{X}$ such that the semigroup $\mathcal{T}(t)$ decomposes as
\begin{equation*}
\mathcal{T}(t)=\mathcal{T}_*(t) + Q(t),
\end{equation*}
where
\begin{equation*}
\mathcal{T}_*(t)=e^{\lambda_*\, t}\sum_{j=0}^{k_*-1}\frac{t^j}{j!}\left(\mathcal{A}+\mathcal{B}+\mathcal{C}-\lambda_*\right)^j\, P_*,
\end{equation*}
and
\begin{equation*}
||Q(t)||\le M_\delta e^{(\lambda_*-\delta) t},\quad \text{for some}\quad \delta>0,\, M_\delta\,\ge 1,\quad \forall\, t\ge 0.
\end{equation*}
\item[(iii)]
If in addition, the irreducibility condition in Proposition \ref{irreducible} holds true, then there exists a rank one operator $P_*$ such that
\begin{equation*}
\displaystyle\lim_{t\to \infty} e^{-\lambda_*\,t}\,\mathcal{T}(t)=P_*.
\end{equation*}
\end{itemize}
\end{theorem}

\section{Further spectral properties}

\begin{enumerate}
\item The proof of Proposition \ref{quasi-compact} in fact shows that $\omega_{ess}(\mathcal{T})\le -\kappa$, i.e. the essential spectrum of $\mathcal{T}(t)$ is contained in the left half-plane $\left\{\lambda\in\mathbb{C}\,|\,\text{Re}(\lambda)\le -\kappa\right\}$.
\item It is worthwhile to note that there is also a different concept of compactness of a semigroup. In particular, a semigroup $\mathcal{T}(t)$ is called essentially compact, if $\omega_{ess}(\mathcal{T})<\omega_0(\mathcal{T})$ holds, i.e. the essential growth bound is strictly less than the growth bound of the semigroup (but possibly both being positive), see e.g. \cite[Ch.V]{EngelNagel}. For essentially compact (but not necessarily quasi-compact) semigroups a similar result to Theorem \ref{quasi-compact-char} can be established. Moreover, essential compactness is a necessary condition for balanced exponential growth, as shown in \cite{Thieme1998DCDS}. To establish essential compactness of the semigroup governing model \eqref{Model} one might be able to relax the lower bound we imposed on the mortality. But we also note that in fact in principle it is possible that 
$$-\infty<\omega_{ess}(\mathcal{T})=s(\mathcal{A}+\mathcal{B}+\mathcal{C})=\omega_0(\mathcal{T})<0,$$ in which case the semigroup $\mathcal{T}(t)$ is not essentially compact, but it is quasi-compact. In fact, case (i) in Theorem \ref{mainresult} covers this scenario, see also below for more details. 
Also note that an essentially compact semigroup can be rescaled to obtain $\omega_{ess}(\mathcal{T})<\omega_0(\mathcal{T})<0$, i.e. the rescaled semigroup is quasi-compact.
\item Case (i) in Theorem \ref{mainresult} covers the following, mutually exclusive (and exhaustive list of) possible scenarios. 
\begin{itemize} 
\item The spectrum of $\mathcal{A}+\mathcal{B}+\mathcal{C}$ is empty (in which case by definition we have $s(\mathcal{A}+\mathcal{B}+\mathcal{C})=-\infty$. 
\item The spectrum of $\mathcal{A}+\mathcal{B}+\mathcal{C}$ is not empty, in which case by virtue of positivity $s(\mathcal{A}+\mathcal{B}+\mathcal{C})\in\sigma(\mathcal{A}+\mathcal{B}+\mathcal{C})$. Moreover, we have 
$$s(\mathcal{A}+\mathcal{B}+\mathcal{C})=-\kappa,$$ the semigroup does not necessarily exhibit balanced/asynchronous exponential growth. The semigroup $\mathcal{T}(t)$ is also not necessarily essentially compact, (but it is quasi-compact).
\item The spectrum of $\mathcal{A}+\mathcal{B}+\mathcal{C}$ is not empty, moreover we have  
$$-\kappa<s(\mathcal{A}+\mathcal{B}+\mathcal{C})\in\sigma(\mathcal{A}+\mathcal{B}+\mathcal{C})<0,$$ 
in which case the semigroup exhibits balanced/asynchronous exponential growth, or rather decay. The semigroup $\mathcal{T}(t)$ is both essentially compact and quasi-compact.
\end{itemize}
Case (ii) in Theorem \ref{mainresult} covers the following qualitatively different cases. 
[Note that in both cases below, $\lambda_*=s(\mathcal{A}+\mathcal{B}+\mathcal{C})$ is necessarily an isolated eigenvalue of finite algebraic multiplicity.]
\begin{itemize}
\item $\lambda_*=0$, and there exists a family of steady states of model \eqref{Model}, and solutions of model \eqref{Model} approach this (possibly multidimensional) subspace of $\mathcal{X}_+$. 
\item  $\lambda_*>0$, and the semigroup $\mathcal{T}(t)$ exhibits balanced exponential growth (with a possibly multi-dimensional global attractor). If in addition the semigroup $\mathcal{T}(t)$ is irreducible, then it exhibits asynchronous exponential growth with a one dimensional globally attracting subspace, which is spanned by the so-called final-size distribution.
\end{itemize}
\end{enumerate}

\section{Conclusion}

We introduced and studied the asymptotic behaviour of a linear structured population model, which we formulated on the space of Radon measures. Our model exhibits two key phenomena: a distributed recruitment process and an unbounded individual state space; both of these give rise to different challenges in the spectral analysis of the problem. Nevertheless, we characterised the asymptotic behaviour of solutions using techniques form the theory of strongly continuous positive semigroups. In particular we established the existence of a finite dimensional global attractor.

Next we are going to use our model \eqref{Model} in the context of specific applications, and therefore naturally we are going to introduce different types of nonlinearities in model \eqref{Model}. The results on the linear semigroup we presented here then will allow us to study qualitative questions also of the nonlinear model, such as existence of positive steady states, utilising also the framework recently developed in \cite{CF}.

%\section{Acknowledgements}

\end{document}